\documentclass[english]{article}
\usepackage{amsthm} 
\usepackage{amsmath}
\usepackage{amssymb} 
\usepackage{quiver}
\usepackage{xspace}
\usepackage{esint} 
\usepackage{stmaryrd}
\usepackage{thmtools} 
\usepackage{mathtools} 
\usepackage{xparse} 
\usepackage{float} 
\usepackage{hyperref} 
\usepackage{xcolor}
\usepackage{todonotes} 

\newtheorem{definition}{Definition}
\newtheorem{theorem}{Theorem}

\newtheorem{lemma}{Lemma}
\newtheorem{proposition}{Proposition}

\newtheorem{remark}{Remark}

\newcount\foreachcount

\makeatletter
\let\ea\expandafter
\def\foreachLetter#1#2#3{\foreachcount=#1
  \ea\loop\ea\ea\ea#3\@Alph\foreachcount
  \advance\foreachcount by 1
  \ifnum\foreachcount<#2\repeat}
\def\definecal#1{\ea\gdef\csname c#1\endcsname{\ensuremath{\mathcal{#1}}\xspace}}
\foreachLetter{1}{27}{\definecal}

\makeatletter
\let\ea\expandafter
\def\foreachLetter#1#2#3{\foreachcount=#1
  \ea\loop\ea\ea\ea#3\@Alph\foreachcount
  \advance\foreachcount by 1
  \ifnum\foreachcount<#2\repeat}
\def\definecal#1{\ea\gdef\csname b#1\endcsname{\ensuremath{\mathbf{#1}}\xspace}}
\foreachLetter{1}{27}{\definecal}

\newcommand{\R}{\mathbb{R}}

\newcommand{\Cat}{\mathbf{Cat}}
\newcommand{\TwoCat}{\mathbf{2}\textbf{-}\mathbf{Cat}}
\newcommand{\TwoCatic}{\mathbf{2}\textbf{-}\mathbf{Cat}^\ic}
\newcommand{\Set}{\mathbf{Set}}
\newcommand{\Hom}{\mathsf{Hom}}
\newcommand{\eval}{\mathsf{eval}}

\newcommand{\tw}{\mathsf{tw}}

\newcommand{\Optic}{\mathbf{Optic}}
\newcommand{\TwoOptic}{\mathbf{2}\textbf{-}\mathbf{Optic}}

\newcommand{\CoPara}{\mathbf{CoPara}}
\newcommand{\CartLens}{\mathbf{Lens_{Cart}}}
\newcommand{\ClosedLens}{\mathbf{Lens_{Cl}}}

\newcommand{\Dbl}{\mathbf{Dbl}}

\newcommand{\comp}{\fatsemi}
\newcommand{\id}{\text{id}}
\newcommand{\curry}{\text{curry }}

\newcommand{\op}{\mathsf{op}}
\newcommand{\co}{\mathsf{co}}
\newcommand{\ic}{\mathsf{ic}}

\DeclareDocumentCommand \internalBracket { o m m } {
  \IfNoValueTF {#1} {
    [#2, #3]
  }{
    [#2, #3]_{#1}
  }
}

\newcommand{\colim}{\mathsf{colim}}
\newcommand{\El}{\mathsf{El}}
\newcommand{\discr}{\mathsf{discr}}

\newcommand{\view}{\mathsf{get}}
\newcommand{\upd}{\mathsf{put}}
\newcommand{\get}{\mathsf{get}}

\newcommand{\graph}{\mathsf{graph}}
\newcommand{\fw}{\mathsf{fw}}
\newcommand{\bw}{\mathsf{bw}}
\newcommand{\M}{\mathsf{M}}

\newcommand{\piico}{\pi^\ic_{0^*}}

\newcommand{\prb}[2]{\begin{pmatrix}{#1} \\ {#2} \end{pmatrix}}
\newcommand{\pr}[2]{\begin{matrix}{#1} \\ {#2} \end{matrix}}
\newcommand{\OpticHom}[4]{\Bigg(\pr{#1}{#2} \, , \pr{#3}{#4}\Bigg)}

\title{Space-time tradeoffs of lenses and optics via higher category theory}
\author{Bruno Gavranovi\'c}

\begin{document}
\maketitle
\begin{abstract}
  Optics and lenses are abstract categorical gadgets that model systems with bidirectional data flow.
  In this paper we observe that the denotational definition of optics -- identifying two optics as equivalent by observing their behaviour \emph{from the outside} -- is not suitable for operational, software oriented approaches where optics are not merely observed, but built with their internal setups in mind.
  We identify operational differences between denotationally isomorphic
  categories of cartesian optics and lenses: their different
  composition rule and corresponding space-time tradeoffs, positioning them at two opposite ends of a spectrum.
  With these motivations we lift the existing categorical constructions and their relationships to the 2-categorical level, showing that the relevant operational concerns become visible. 
  We define the 2-category $\TwoOptic(\cC)$ whose 2-cells explicitly optics' internal configuration. We show that the 1-category $\Optic(\cC)$ arises by locally quotienting out the connected components of this 2-category.
We show that the embedding of lenses into cartesian optics gets weakened from a
functor to an oplax functor whose oplaxator now detects the different
composition rule.
We determine the difficulties in showing this functor forms a part of an
adjunction in any of the standard 2-categories.
We establish a conjecture that the well-known isomorphism between cartesian lenses
and optics arises out of the lax 2-adjunction between their double-categorical
counterparts.
In addition to presenting new research, this paper is also meant to be an accessible introduction to the topic.
\end{abstract}

%

	\tableofcontents
	\section{Introduction}

Lenses and optics are have recently received a great deal of attention from the applied category theory community.
They are abstract data structures that model systems exhibiting bidirectional data
flow.
There's a number of disparate places they've been discovered in: deep learning
(\cite{GradientBasedLearning, LensesAndLearners}), game theory
(\cite{DiegeticOpenGames, CompositionalGameTheory}), bayesian learning
(\cite{Bayesianlearning, BayesianOG}), reinforcement
learning (\cite{ValueIterationIsOpticComposition}), database theory
(\cite{FunctorialAggregation}), dynamical systems (\cite{OpenDynamicalSystems}),
data accessors (\cite{ModularDataAccessors}), trading protocols (\cite{Escrows}),
server operations (\cite{LensesForComposableServers}) and more (\cite{LensResources}).

As evident by the breadth of their applications, lenses and optics don't assume
that the underlying systems are of any particular kind.
Instead, they are defined parametrically for some base category $\cC$, which is
only required to satisfy a minimal set of axioms.
By appropriately instantiating this category we can recover various kinds of systems -- deterministic, probabilistic, differentiable, and so on.
This makes it possible to treat the bidirectionality in an abstract way, proving
theorems about whole classes of bidirectional processes that satisfy particular properties.

The two constructions we focus on in this paper -- cartesian lenses and optics\footnote{We note that there is a whole \emph{zoo} of bidirectional gadgets,
  each with their own kind of behavior, and their own abstract interface that
  the underlying world needs to satisfy. A lot of effort has been put in towards
  representing all of these constructions in an unifying way (see any of
  \cite{DependentOptics, FibreOptics, ProfunctorOptics,
    GeneralisedLensCategories})} -- are related, but have different requirements on the base category. 
For optics to be defined, we require the base category to permit parallel
composition of processes, i.e. a monoidal structure. 
For lenses we additionally require this structure to be cartesian, i.e.
the ability to coherently copy and delete information.
These play well together -- defining optics in a cartesian category gives us a
category isomorphic to lenses, as it is well-established in the literature.

In this we paper observe that this isomorphism is denotational in nature and blind
to operational concerns relevant to their practical implementations.
Namely, it treats optics extensionally -- describing them as being observed \emph{from the outside}. This means that any matters of
their internal setup, especially ones relevant to making a distiction between an efficent
and an inefficient implementation, are ignored.
But in a modern, software oriented world we're not merely observing these optics
from the outside in -- we're instead building them from the inside out.
We're choosing their particular internal states, and in most cases we don't have the
luxury of \emph{not} distinguishing between an efficient and an inefficient representation, as often only the former can compute an answer for us.
As the current categorical framework doesn't have a high-enough resolution to
formally capture these distinctions, we seek to provide one. We lift the existing 1-categorical formalism to a 2-categorical one. We show how to track and manipulate the internal state of optics, making a distinction between denotationally equivalent, but
operationally different kinds of optics.


We first start out by unpacking the category of lenses and emphasizing that a seldom talked about perspective: that
lenses are cartesian optics with one option removed: the option to choose the type of its internal state.
We show how this lack of a tangible way to refer to this important notion has
significant operational consequences when composing lenses.
Namely, we'll see that lenses implement a particular kind of a space-time
tradeoff that's in the deep learning literature called \emph{gradient checkpointing}.
We move on to unpacking the category of optics and notice they they have a different composition rule than lenses, implementing a different space-time tradeoff.
What motivates the rest of the paper is the observation that the difference is
completely invisible to the categorical machinery.

We observe that optics are defined by a particular kind of colimit, suggesting
an avenue forward by instead defining them as an \emph{oplax} colimit.
We do so, and thus define the \emph{2-optics}: a 2-category whose 2-cells now explicitly track their internal state.
We show that they coherently reduce to 1-optics by locally computing their
connected components.

We then go on to explore the isomorphism $\CartLens(\cC) \cong \Optic(\cC)$ in
this 2-categorical setting.
We show that the 2-categorical setting now locally hosts an adjunction between
the corresponding categories, and that the embedding of lenses into cartesian optics is upgraded from a functor to an oplax functor whose oplaxator now detects the different
composition rule.
We determine the difficulties in showing the oplax functor forms in any of the
standard categories whose 1-cells are lax functors.
We establish a conjecture that the well-known isomorphism between cartesian lenses
and optics arises out of the lax 2-adjunction between their double-categorical
counterparts, as an image under the local connected components quotient.

\subsection{Acknowledgements.}

We thank Igor Bakovi\'c, Fosco Loregian, Mario Rom\'an, Matteo Capucci and Jules Hedges for helpful conversations.

\subsection{Notation.}

We write morphisms in diagrammatic order, so composition of $A \xrightarrow{f}
B$ and $B \xrightarrow{g} C$ is written as $A \xrightarrow{f \comp g} C$. We write $\pi_2 : A \times B \to B$ for the projection in the 2nd component.

	\section{Cartesian Lenses}
\label{sec:cart_lenses}

Lenses come in many shapes and sizes.
In this paper, we tell the story from the point of view of \emph{cartesian}\footnote{As opposed to \emph{closed}.} lenses (\cite{BimorphicLenses, CategoriesOfOptics, ProfunctorOptics}).
This is a category which we denote by $\CartLens(\cC)$.
We proceed to unpack its contents and, as most of the content of this paper is motivated by its previously unnoticed operational aspects, we take special care in doing so.
We look out for any potentially resource-relevant operations such as data copying or recomputation.

$\CartLens(\cC)$ can be defined for any base category $\cC$ which is cartesian monoidal.
Its objects are pairs of objects in $\cC$, denoted by $\prb{A}{A'}$, where we interpret the value of type $A$ as going \emph{forward} and value of type
$A'$ as going \emph{backward}.



A morphism $\prb{A}{A'} \to \prb{B}{B'}$ is a \emph{cartesian lens}.
It consists of a two morphisms in $\cC$, $\view : A \to B$ and $\upd : A \times
B' \to A'$, roughly thought of as the \emph{forward} and the \emph{backward}
part of a lens.\footnote{Sometimes the terminology $\mathsf{view}$ and $\mathsf{update}$ is used instead of $\view$ and $\upd$.}
We can visualise lenses graphically using the formalism of \emph{string diagrams}
\cite{Selinger} (Figure \ref{fig:lens}), an especially useful visual language for studying the flow of information in a lens.

\begin{figure}[H]
  \centering
  \includegraphics[width=.7\textwidth]{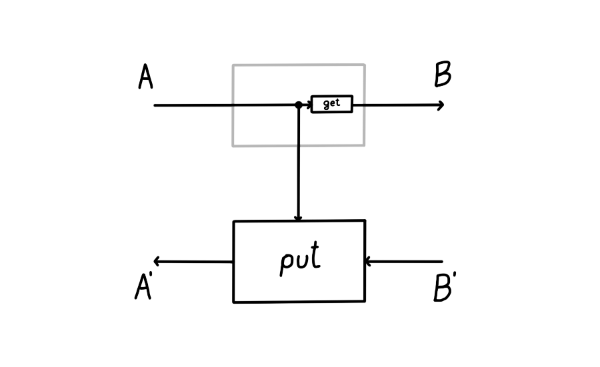}
  \caption{Graphical depiction of a lens. This particular kind
    of a string diagram arises out of string diagrams for optics (\cite{StringDiagramsForOptics}).}
  \label{fig:lens}
\end{figure}
The flow of information works as follows. Information starts at the input
of type $A$. The lens takes in this input and produces two things: a copy of it
(sent down the vertical wire, where the operation of copying is drawn as a black
dot) and the output $B$ (via the $\view$) map. This is the \emph{forward pass} of the lens, and is drawn with the gray outline.
Then, the environment takes this output $B$ and turns it into a response $B'$
(not drawn).
This lands us in the \emph{backward pass} of the lens. Here the lens via the
$\upd$ map consumes two things: the response $B'$ and the previously saved
copy of the input on the vertical wire, turning them back into $A'$.

A lens has an inside and an outside. The outside are the ports $(A, A')$ and
$(B, B')$. These ports are the interface to which other lenses connect.
The inside is the vertical wire whose type is $A$. The vertical wire is the
internal state of the lens (sometimes also called \emph{the residual}) --
mediating the transition between the forward and the backward pass.

\begin{figure}[H]
  \centering
  \includegraphics[width=.7\textwidth]{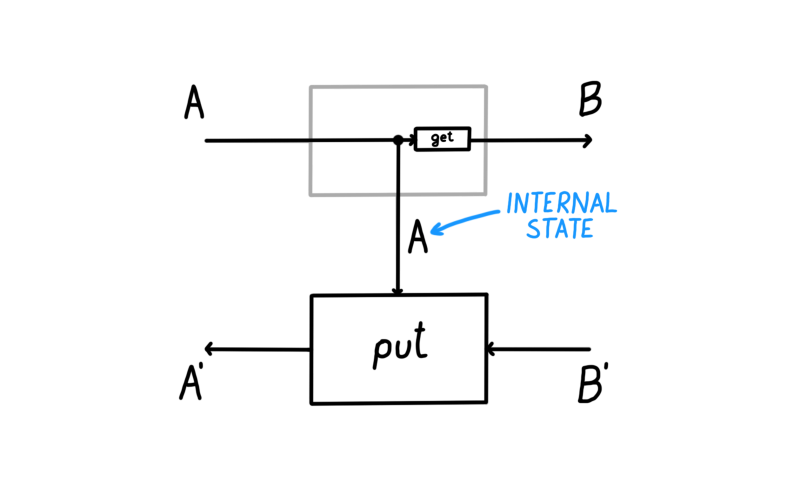}
  \caption{The internal state of the lens}
  \label{fig:lens_residual}
\end{figure}

Here we are explicitly referring to the internal state because of what's to come,
but we emphasize that in the lens literature this concept hasn't been reified.
In the lens literature the internal state is not explicit data that can be
manipulated, and is instead being implicitly threaded through definitions and
theorems -- always being pegged to the forward part of the domain of a lens.
More precisely,
  \textbf{the type of the internal state of a lens $\prb{A}{A'} \to \prb{B}{B'}$
    is always equal to $A$.}
In what follows, we will see how this lack of a tangible way to refer to this
important notion has significant operational consequences when composing lenses.

\begin{center}
  \label{motto}
  "The simplicity of the presentation of lenses is balanced by the complexity of
  their composition."
\end{center}

Armed with the above motto we proceed to unpack the definition of lens composition.
Suppose we have two lenses: $(A, A') \xrightarrow{(\view_1,
  \upd_1)} (B, B')$ and $(B, B') \xrightarrow{(\view_2, \upd_2)} (C,
C')$, as drawn below in Figure \ref{fig:lens_precomposed}.

\begin{figure}[H]
  \centering
  \includegraphics[width=.9\textwidth]{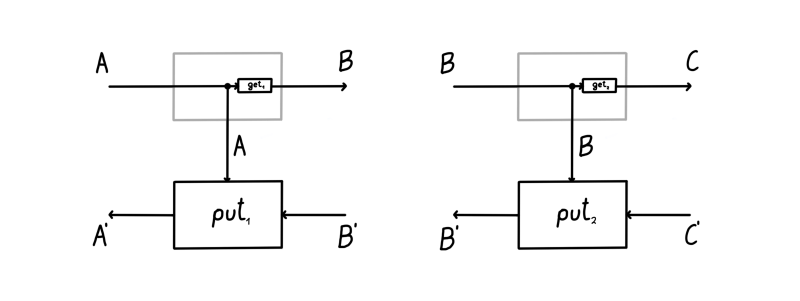}
  \caption{Two composable lenses.}
  \label{fig:lens_precomposed}
\end{figure}

Using the grapical formalism of the figure above, it seems reasonable
to define the composite of these two lenses simply by plugging them along the two
matching ports $\prb{B}{B'}$. We draw the result of this in Figure \ref{fig:optic_composed}.
We invite the reader to ponder this definition before moving on.
Is this composition well-defined?

\begin{figure}[H]
  \centering
  \includegraphics[width=.9\textwidth]{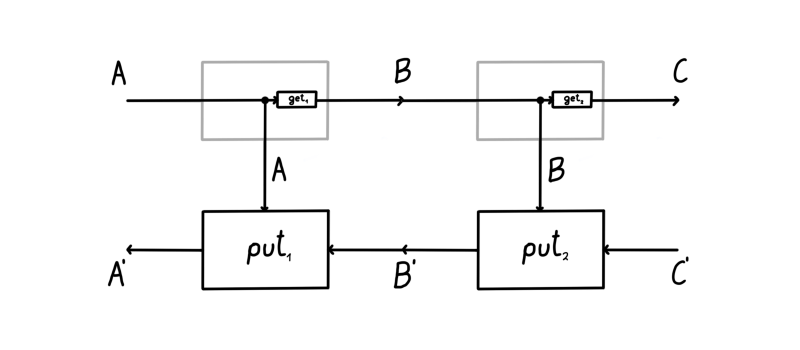}
  \caption{First guess at a possible lens composition.}
  \label{fig:optic_composed}
\end{figure}

The answer is no! What is drawn above is not a lens. It turns out that this elegant and
plausible looking solution has an issue. Namely, if we look at the figure, we see that the internal state of this supposed lens
$\prb{A}{A'} \to \prb{C}{C'}$ is $A \times B$.
But we've previously established that the type of the internal state of every lens with
domain $\prb{A}{A'}$ is pegged to $A$ itself, as internal state is not data available for manipulation.
This means that what we've defined above is \emph{some kind} of a bidirectional
process, but not a lens.
Another way to see this is to try to write down the $\view$ and $\upd$ maps explicitly.
Below we explicitly do so -- we write out the \emph{correct} definition of lens composition (forgetting the above image for a moment).

\begin{definition}[Lens composition]
  Consider two lenses:\\
  \begin{center}
    $\prb{A}{A'} \xrightarrow{\prb{\view_1}{\upd_1}} \prb{B}{B'}$ and
    $\prb{B}{B'} \xrightarrow{\prb{\view_2}{\upd_2}} \prb{C}{C'}$.
  \end{center}
  Their composite $\prb{A}{A'} \xrightarrow{\prb{\view}{\upd}} \prb{C}{C'}$ is
  defined as:

  \begin{align*}
    \view &\coloneqq A \xrightarrow{\view_1} B \xrightarrow{\view_2} C \\
    \upd &\coloneqq A \times C' \xrightarrow{\graph(\view_1) \times C'} A \times B \times C' \xrightarrow{A \times \upd_2} A \times B' \xrightarrow{\upd_1} A'
  \end{align*}
\end{definition}

While the definiton of the composite $\view$ is simple, the composite $\upd$
is more complex.\footnote{We observe that $\upd$ defines a generalised form of
  chain rule \cite[p. 11]{GradientBasedLearning}} If we look at at $\upd$, we see first see $\graph(f)$ applied
to $A$, which copies the input and applies $f$ to one of
the copies.\footnote{We write out the formal definition of $\graph(f)$ in Def.
  \ref{def:graph}.}
That copy results in a $B$, which is used in $\upd_2$.
The map $\upd_2$ gives us a $B'$ which is used together with the other copy of $A$ to
obtain an $A'$ using $\upd_1$.
Observe that this is the only way lens composition can be defined. \footnote{This
  can be seen by the reasoning going \emph{backwards}: "the only thing that we can use to produce $A'$
  is $\upd_1$, and the only way to produce its inputs is by..."},
Lens composition is shown graphically in Figure \ref{fig:lens_composed}.

\begin{figure}[H]
  \centering
  \includegraphics[width=.9\textwidth]{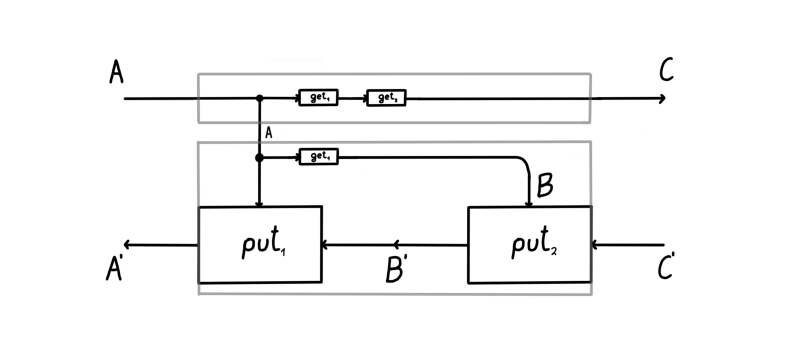}
  \caption{Composition of two lenses}
  \label{fig:lens_composed}
\end{figure}

We can immediately observe that this is different than our original guess: 1) there are two $\view_1$ maps, and 2) the input $A$ is copied twice, not once.
With the hindsight that we're interested in implementing these lenses in
software, the fact that some functions are computed twice raises some suspicions
about the feasibility such an implementation.
To get a better sense of what's going on, we up the stakes and depict a
composition of \emph{three} lenses in Figure \ref{fig:lens_composed_three}.

\begin{figure}[H]
  \centering
  \includegraphics[width=.9\textwidth]{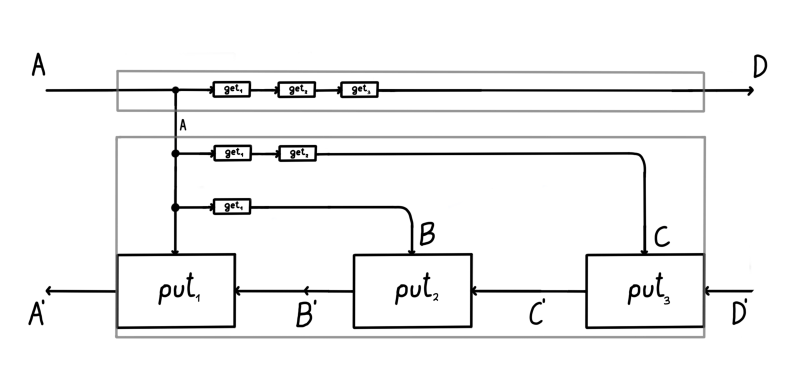}
  \caption{Composition of three lenses.}
  \label{fig:lens_composed_three}
\end{figure}

At this point things start to look crowded. There are now 6 $\view$ maps.
We are also copying $A$ three times in total. In general, it seems that
composing more lenses only exacerbates the problem. What is going on?

If we look closely, we see that the backward pass of this composite for
each $\{\upd_i\}_{i=1}^3$ map \emph{independently} computes from scratch what that  map needs.
For instance, $\upd_3$ uses $\view_1 \comp \view_2$ to compute its internal state, and $\upd_2$ uses $\view_1$, while $\upd_1$ uses just the already available $A$,
but none of these computations share results of computation with each
other.

This strategy of recomputing every intermediate result from scratch might
certainly seem disadvantageous, but we observe that it's a part of a tradeoff:
 this strategy uses less memory.
Only the initial state $A$ needs to be preserved in memory, and everything needed
from the backward pass can be computed from it. It is also never the case that
both $\view_1 \comp \view_2$ and $\view_1$ in the backward pass of Figure
\ref{fig:lens_composed_three} need to be computed in parallel the same time
(which would require more memory): it's necessary to compute the output of the former
before the output of the latter can be used.\footnote{To help with intuition, we invite the reader to have a look at the \emph{animation} of this process, available at the
  \href{https://twitter.com/bgavran3/status/1563628021692743680}{following
    link}.}

This means that lens composition picks a particular space-time tradeoff when
solving the issue of backpropagating information.
It uses \emph{less space} (as it doesn't need to save intermediate states of computation in memory), but \emph{more time} (as it needs to recompute data).

\begin{remark}
  This kind of space-time tradeoff has a name in the deep learning and
  automatic differentation community: it's called \emph{gradient
    checkpointing}\cite{Checkpointing, SublinearCost}. It is often used with
  very large neural networks where storing all the intermediate results is
  prohibitive memory wise, or when available computation resources
  are constrained memory-wise.
  While it is understood that lenses are intricately tied to the chain rule, to the best of our knowledge this is the first time the connection between lenses and gradient checkpointing has been established.
\end{remark}

The explanation of why the structure of lenses ended up implementing this particular tradeoff can be seen in Figure \ref{fig:bottleneck}, showing a composite of two lenses.
Here we see the residual $A$ circled in blue mediating the passage from the forward pass to the backward pass.
Observe that all the data communicated between the forward and the backward pass has to be squeezed through this $A$-shaped hole.

\begin{figure}[H]
  \centering
  \includegraphics[width=.9\textwidth]{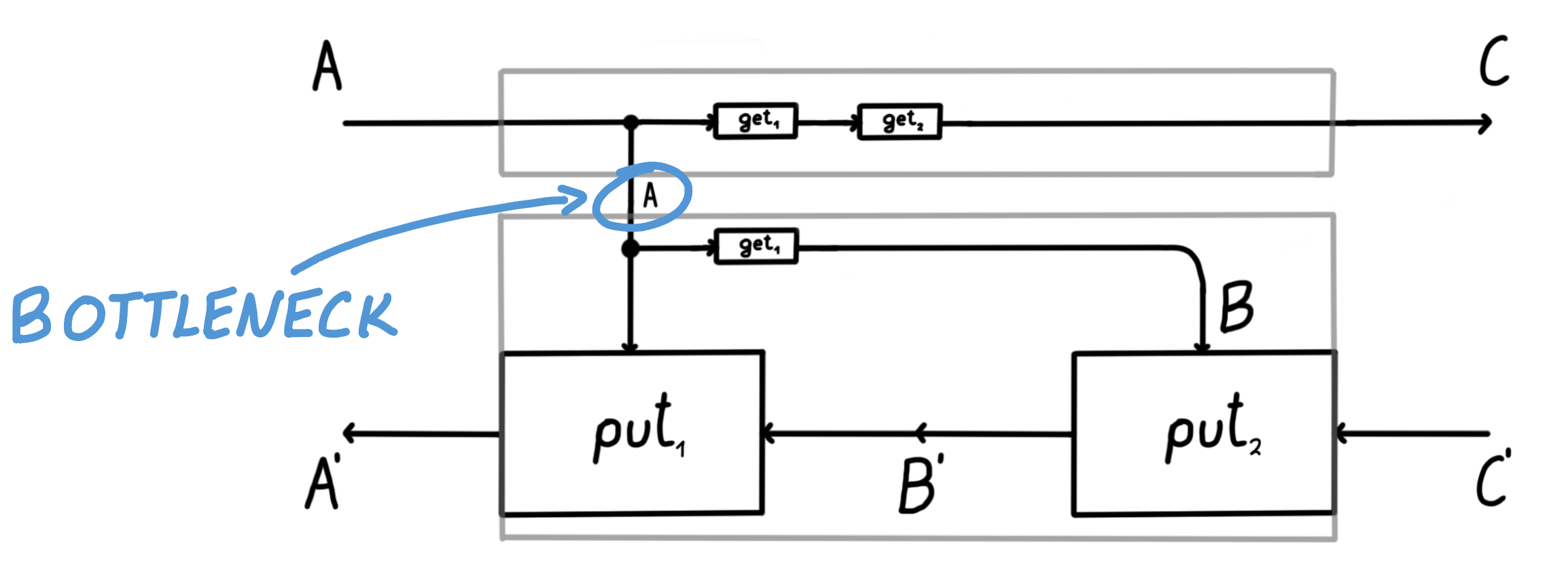}
  \caption{There is a bottleneck in every lens $\prb{A}{A'} \to \prb{B}{B'}$, necessitating that all the data required for the backward pass is sequeezed through an $A$-shaped hole.}
  \label{fig:bottleneck}
\end{figure}

This means that intermediate state of type $B$ that $\get_1$ computed in the
forward pass can't be communicated to the backward pass.
Instead, there is no other way, but for the backward pass to separately
recompute this information from $A$.
\footnote{If we have a composite of 3, 4, or in general $n$ more lenses, then there is $n
- 1$ levels of separate computation, and each level $i$ is has a sequence of
$\get$ maps of length $i$ composed. The memory required to compute gradients is
in our graph is constant in the number of layers $n$, but the number of node evaluations scales with $n^2$.} 
And this itself arises precisely because when in defining a lens we have no freedom to
choose the type of data that will be communicated from the forward pass to the
backward pass. 

\begin{remark}
This phenomenon seems to have first been observed in \cite[Section
3.1.]{SimpleAD} where the author described their initial attempts to efficiently
compute reverse-mode derivatives with cartesian lenses, only to identify the
aforementioned redundancy problems. He went on to propose a solution using
\emph{closed} lenses, something we touch upon in Remark \ref{rem:closed_lax}.
Interestingly, the author never used the term \emph{lens} in the entire paper.
\end{remark}

\subsection{Where to?}

We've established that lens composition implements a generalised form of chain
rule in a manner that uses less space but more time. This is a result of the absence of an explicit way to refer to the internal state of lenses.
Two questions now become natural to ask: a) How can we recover other space-time
tradeoffs? and b) How can we explicitly refer to and manipulate this internal state?
In next section we answer both of these questions with \emph{optics}.


	\section{Optics}

The category of optics is a generalisation of the category of lenses, and has
been thoroughly studied in the literature \cite{ProfunctorOptics,
  CategoriesOfOptics, StringDiagramsForOptics, ModularDataAccessors}.
Unlike lenses, optics do not require a cartesian structure and can instead be
defined for any base category that is merely monoidal.\footnote{This is not the
  most general definition of optics, see \cite{DependentOptics, FibreOptics, ProfunctorOptics}.}
We denote this category by $\Optic(\cC)$ and proceed to unpack its contents.
Its objects of are pairs of objects in $\cC$, just like with
lenses. However, differences start appearing once we start looking at the morphisms.

\begin{definition}[{\cite[Def. 2.0.1.]{CategoriesOfOptics}}]
  \label{def:homset_optic}
  The set of optics $(A, A') \to (B, B')$ is defined as the following coend
  \[
    \Optic(\cC)\OpticHom{A}{A'}{B}{B'} \coloneqq \int^{M : \cC} \cC(A, M \otimes B)
    \times \cC(M \otimes B', A')
  \]

  Its elements are equivalence classes of triples $(M, f, f')$, where $M
  : \cC$, $f : A \to M \otimes B$ and $f' : M \otimes B' \to A'$. They're
  quotiented out by the equivalence relation where $(M, f, f') \sim (N, g, g')$
  if there is a residual morphism $r : M \to N$ in $\cC$ such that the following diagrams commute:
\begin{equation}
  \label{eq:optic_equiv_diagrams}
  \begin{tikzcd}[row sep=3ex, column sep=4ex]
      A && {M \otimes B} && {M \otimes B'} && {A'} \\
      \\
      && {N \otimes B} && {N \otimes B'}
      \arrow["f", from=1-1, to=1-3]
      \arrow["{r \otimes B}", from=1-3, to=3-3]
      \arrow["g"', from=1-1, to=3-3]
      \arrow["{f^\sharp}", from=1-5, to=1-7]
      \arrow["{g^\sharp}"', from=3-5, to=1-7]
      \arrow["{r \otimes B'}"', from=1-5, to=3-5]
    \end{tikzcd}
\end{equation}
\end{definition}

This definition might look daunting, but it is the result the dualisation of Motto
\ref{motto} whose consequence will be a more straightforward definition of
composition.
Nonetheless, we will see that each part of the definition has intuitive meaning.
An optic $(\M, \fw, \bw)$ has three components.
The object $\M$, the type of the internal state, the forward map $\fw$, and the
backward map $\bw$.
The shape of an optic is drawn in Figure \ref{fig:optic_residual}, and it has a similar data flow as a lens.
It takes in some $A$ in the forward pass, and using the the map $\fw$ it
produces the product $M \otimes B$, for the chosen type $M$.
The environment then takes in the $B$ and responds with a $B'$, allowing the
backward part to use $M \otimes B'$ and produce $A'$.

\begin{figure}[H]
  \centering
  \includegraphics[width=.6\textwidth]{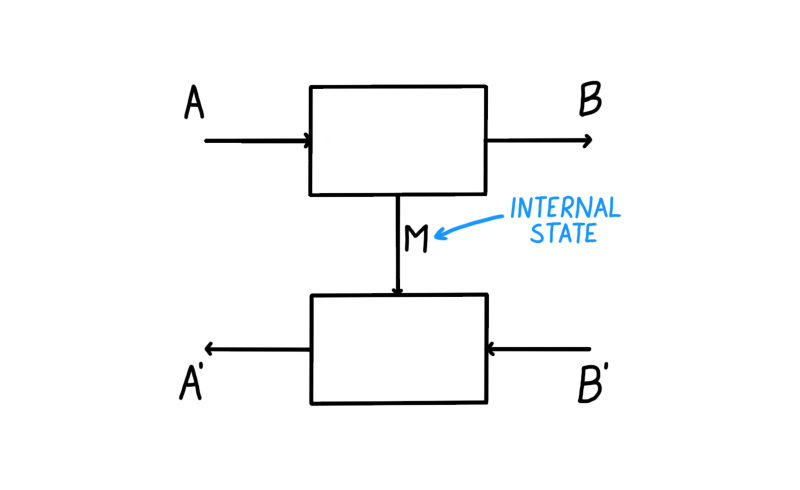}
  \caption{In defining an optic we have the freedom to choose the type of
    internal state, a liberty not available with lenses.}
  \label{fig:optic_residual}
\end{figure}

The last component of the definition is the equivalence relation.
It's a formal description of the idea that we think of optics sa being \emph{observed from the outside}.
This means that the type of the internal state of an optic isn't externally available information.
This in turn means that there is no way to distinguish between two optics that
have the same extensional behavior\footnote{In other words, the only way it's
  possible to distinguish between two optics is if there is some
  input $A$ and some environment response $B \to B'$ such that two optics
  produce a different $A'$.}, but different types of internal states.

\begin{remark}\label{rem:eq_symmetric}
  The directionality of the residual morphism in the equivalence relation does not matter.
  Because equivalence relations are symmetric, given $(M, f, f')$
and $(N, g, g')$, both morphisms $r : M \to N$ and $r' : N \to M$
(making the appropriate diagrams commute) induce an equivalence relation $(M, f,
f') \sim (N, g, g')$.
\end{remark}

Now we move on to describing the relationship between lenses and optics. 

\subsection{Cartesian Lens - Cartesian Optic isomorphism}
\label{subsec:lens_optic_iso}

If the base $\cC$ of optics is \emph{cartesian} monoidal, the resulting
cartesian optics are isomorphic to cartesian lenses.

\begin{proposition}
  \label{prop:lenseqoptic}
  When $\cC$ is cartesian monoidal, we have $\CartLens(\cC) \simeq \Optic(\cC)$.
\end{proposition}

Unlike the slick proof of this proposition ({\cite[Prop. 2.0.4.]{CategoriesOfOptics}},
which we refer the reader to), here we take special care in unpacking the
non-trivial action of this isomorphism on hom-sets.
We first see how turning a lens into an optic reifies the internal state,
allowing us to explicitly track and manipulate it.
The reason we do this is Remark \ref{rem:residual_morphism_nontrivial} which we
will see is a shadow of a higher-categorical construction we will see in Section \ref{sec:two_optics}.

\begin{proposition}[Cartesian lenses $\to$ cartesian optics]
  \label{prop:lens_as_optic}
  For a cartesian category $\cC$, and every pair of objects $\prb{A}{A'}$ and
  $\prb{B}{B'}$ there is a function $R$ reifying the residual of a lens defined as
  \begin{align*}
    \CartLens(\cC)\OpticHom{A}{A'}{B}{B'} &\xrightarrow{R} \Optic(\cC)\OpticHom{A}{A'}{B}{B'}\\
              \prb{f}{f'} &\mapsto (A, \graph(f), f')
  \end{align*}
\end{proposition}

This finally gives us justification for the choice of the graphical language
used to draw lenses, where lenses were previously drawn
in their suggestive optic representation.
Just as in Figure \ref{fig:lens}, we see that a) the residual of the resulting optic is set to $A$, and b) the input $A$ is copied before being sent down as the residual.
Conversely, starting from an optic we can always erase the residual, and
``normalise'' the optic into its lens representation.

\begin{proposition}[Cartesian optics $\to$ cartesian lenses]
  \label{prop:optic_as_lens}
  For a cartesian category $\cC$, and every pair of objects $\prb{A}{A'}$ and
  $\prb{B}{B'}$ there is a function $E$ erasing the residual of a lens defined as
  \begin{align*}
    \Optic(\cC)\OpticHom{A}{A'}{B}{B'} &\xrightarrow{E} \CartLens(\cC)\OpticHom{A}{A'}{B}{B'}\\
    (\M, \fw, \bw) &\mapsto \prb{\fw \comp \pi_2}{((\fw \comp \pi_1) \times B') \comp \bw}
  \end{align*}
\end{proposition}

\begin{remark}
  \label{rem:residual_morphism_nontrivial}
The proof that $R \comp E = \id$ is trivial, but going the other way it isn't.
Showing that $E \comp R = \id$ involves showing $(\M, \fw, \bw)$ is
equivalent to $(A, \graph(\fw \comp \pi_2), ((\fw \comp \pi_1) \times B') \comp
\bw)$ which involves exhibiting a non-trivial witness for the equivalence: the
residual morphism $\fw \comp \pi_1 : A \to M$.
\end{remark}


\subsection{How do optics compose?}

As previously suggested, optic composition has a different definition than lens composition.

\begin{definition}[Optic composition, {\cite[page 5.]{CategoriesOfOptics}}]
  Consider two optics:\\
  \begin{center}
    $\prb{A}{A'} \xrightarrow{(\M_1, \fw_1, \bw_1)} \prb{B}{B'}$ and
    $\prb{B}{B'}
    \xrightarrow{(\M_2, \fw_2, \bw_2)} \prb{C}{C'}$.\\
  \end{center}
  We define their composite $(\M, \fw, \bw)$ as
  \begin{align}
    \M  &\coloneqq \M_1 \otimes \M_2 \label{eq:optic_comp_res}\\
  \fw &\coloneqq A \xrightarrow{\fw_1} M_1 \otimes B \xrightarrow{M_1 \otimes \fw_2} M_1 \otimes M_2 \otimes C\\
  \bw &\coloneqq M_1 \otimes M_2 \otimes C' \xrightarrow{M_1 \otimes \bw_2} M_1 \otimes B' \xrightarrow{\bw_1} A'
\end{align}

\end{definition}

This is essentially a composition of coparameterised maps in the forward pass,
and a composition of parameterised maps in the backward pass \cite{FibreOptics}.
The above symbolic description has a simple pictorial one: we simply draw a box
around the individual optics (Figure \ref{fig:optic_composition_box}).\footnote{We invite the reader to also have a look at the \emph{animation} of
  the optic composition, available at the author's blog post
  \href{https://www.brunogavranovic.com/posts/2022-02-10-optics-vs-lenses-operationally.html}{Optics
    vs. Lenses, Operationally}.}

\begin{figure}[H]
  \centering
  \includegraphics[width=.7\textwidth]{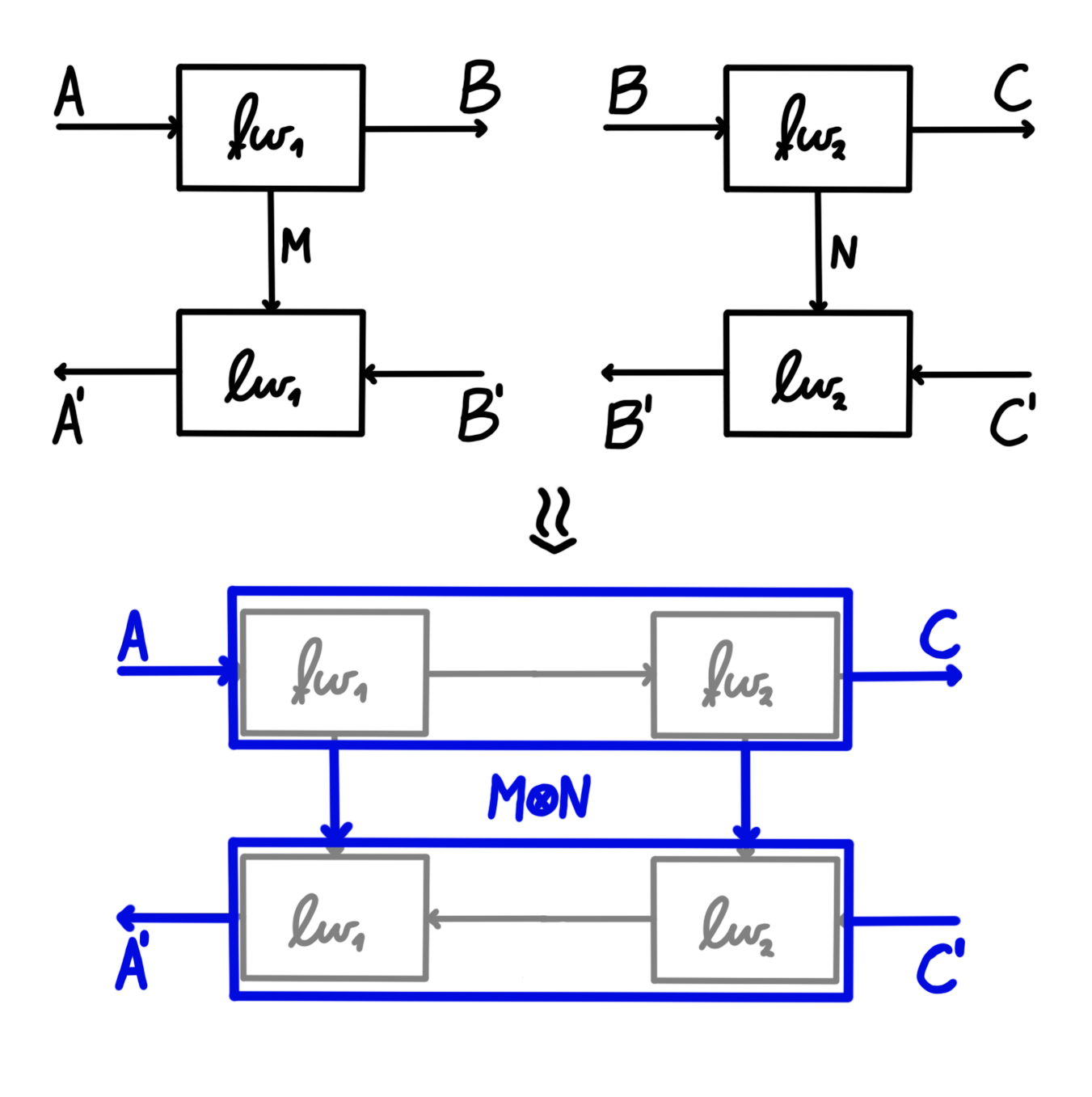}
  \caption{Graphical depiction of composition of two optics.}
  \label{fig:optic_composition_box}
\end{figure}


As originally hinted, optic composition implements a different space-time
tradeoff than lenses.
We notice that the newfound liberty of choosing the type of internal
state on a per-optic basis allows the composition of two optics with residuals
$\M_1$  and $\M_2$ to pick the product $\M_1 \otimes \M_2$ as its internal
state.
This allows optics to break down the problem of saving intermediate state into
smaller pieces: each optics takes care of storing their own data.
In turn, this removes the need to recompute any information, at the expense of
needing more memory.
This tradeoff becomes tricky if memory is a limited resource, as composing optics in a
sequence, causes their residuals to be composed in parallel (Eq.
\ref{eq:optic_comp_res}).
For cartesian optics, this means that the longer our chain of composition is, the more memory we need, something that is not true for lenses.

Now we have seen two different ways to compose these bidirectional gadgets.
Can we formally establish a categorical connection?

\subsection{Two ways to compose?}
\label{subsec:two_ways_to_compose}

The following point motivates the rest of this paper. Say we start with two lenses:
  $\prb{A}{A'} \xrightarrow{\prb{f}{f'}} \prb{B}{B'}$ and $\prb{B}{B'} \xrightarrow{\prb{g}{g'}} \prb{C}{C'}$.
There are two ways to obtain a composite optic, shown in Figure \ref{fig:lens_optic_comp}.
We can either compose them as lenses, and then turn this composition into an optic, or we can first turn these lenses into optics, and then compose the optics.

\begin{figure}[H]
  \centering
  \includegraphics[width=\textwidth]{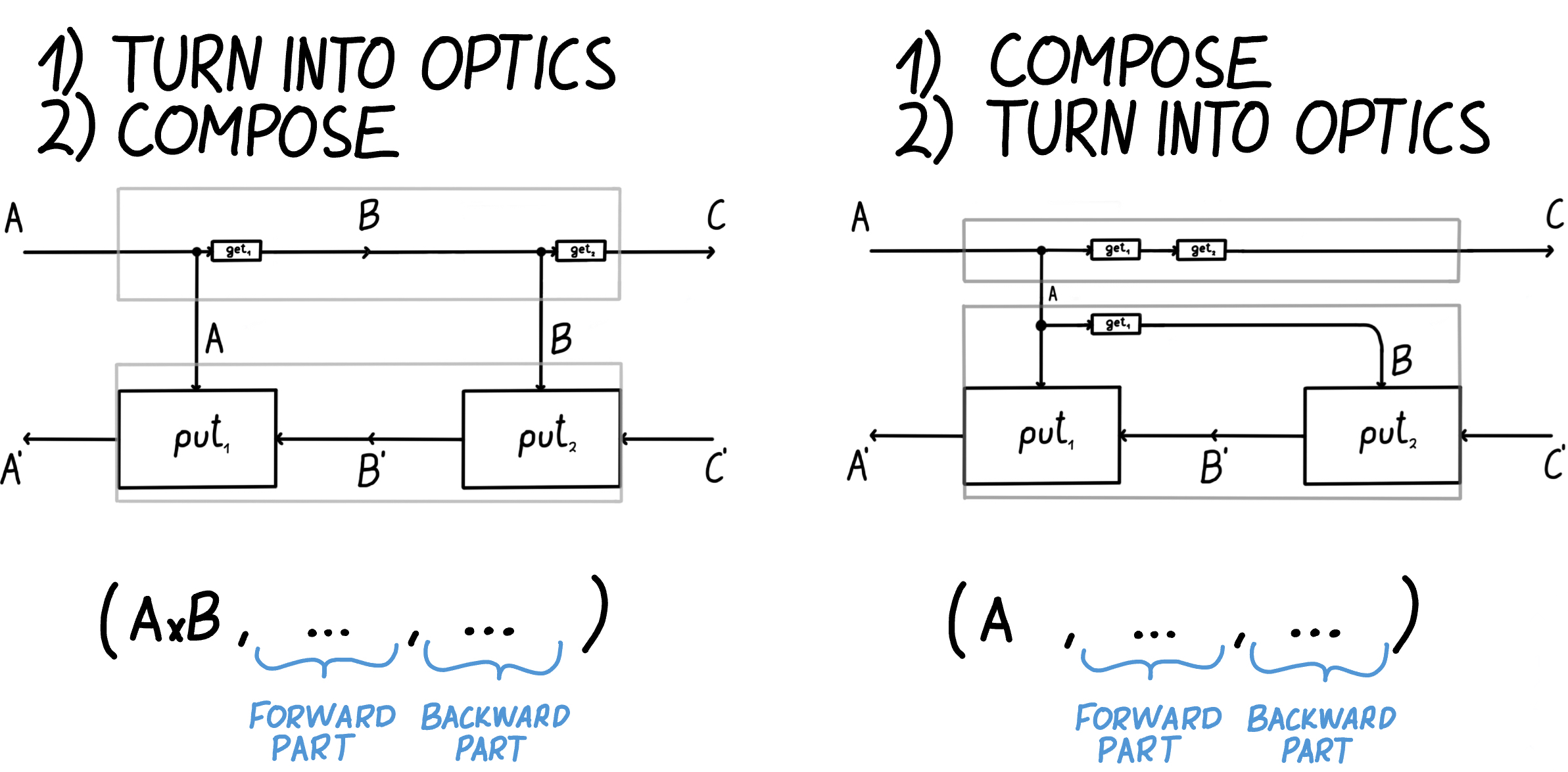}
  \caption[figurelens_optic_comp]{How to turn two composable lenses into an optic? Either by separately turning these lenses into optics and then composing them (left), or
    composing them and then turning the result into an optic (right). Whether these
    are equivalent depends on whether we're taking the denotational (extrinsic)
    or operational (intrisic) point of view.\footnotemark}
  \label{fig:lens_optic_comp}
\end{figure}

\footnotetext{In the figure we omit the cumbersome symbolic description of forward and backward maps of these optics; they can be found in the definition of the oplaxator in Thm. \ref{thm:oplax_embedding}}
%

If we first turn them into optics, and then compose, the result is an optic
whose residual is of type $A \times B$. When implemented,  this optic requires more memory, but less time, as it reuses computation.
Alternatively, if we first compose them and then turn the resulting lens
into an optic, we obtain an optic whose residual is the equal to the type of top-left input $A$.
This optic requires less memory, but more time. 
The question that motivates the rest of this paper is: \textbf{are these optics equivalent?}

  
If we look at the existing categorical framework, we find that the answer is
yes.
This can be seen in a few ways. The most straightforward one is to notice that
asking whether these results are equivalent is asking precisely if the embedding $\CartLens(\cC) \to \Optic(\cC)$ preserves composition (i.e. whether it's a 1-functor).
It's been shown in Prop. \ref{prop:lenseqoptic}) that the answer is yes.
Another way to show this is to exhibit a witness for this equivalence, and
indeed we can: it's the reparameterisation $\graph(f) : A \to A \times
B$.\footnote{Observe that it is crucial the underlying category $\cC$ is
  cartesian, as to prove that the diagrams in Eq. \ref{eq:optic_equiv_diagrams}
  commute we need to slide the $\view$ map through the copy. This will be
  elaborated in detail in Remark \ref{rem:rewrite}.}

We now observe that we have obtained \emph{an} answer to the bolded question,
but not \emph{the} answer.
We've obtained an answer of a particular \emph{denotational} nature.
This answer assumes that we're observing these optics \emph{from the outside},
therefore ignoring any matter of their internal setup.
As there is no way to observe these optics' internal state, there is no way for us to
make a distinction between them.
While this is a valid reference frame, it is not the only one.

In a modern, software oriented world, we're not merely observing these optics
from the ouside in.
We're building them from the inside out. We're choosing their particular
internal states, and it is nature that's actually making a distinction between
them -- by only sometimes only computing the answer if we've chosen the efficient representation for our purposes.

As current categorical framework doesn't have a high-enough resolution to
formally capture these differences, there is a rising need to provide one -- one
that is able to make a distinction between denotationally equivalent, but
operationally different kinds of optics.

In the rest of this paper, we will see how this can be done by using higher
category theory.


	\section{Categorical interlude: (1-) vs. oplax colimits}

In Figure \ref{fig:lens_optic_comp} we have seen that the current coend definition
of optics identifies a lot of information that we would like
to explicitly keep track of.
This happens often in mathematics, where colimits and traditional quotients
identify ``too much'', leading mathematicians to study more refined versions
thereof.
In this interlude we explicitly show how to do this. We first show how to add an extra
level of fidelity by computing \emph{oplax} colimits instead of colimits, and then
showing that the same idea applies to coends, as they're special kind of colimits.

We start out with the well-known monoidal adjunction
\[\begin{tikzcd}[row sep=2ex]
    & {} \\
    \Set && \Cat \\
    & {}
    \arrow["\discr"', curve={height=12pt}, from=2-1, to=2-3]
    \arrow["{\pi_0}"', curve={height=12pt}, from=2-3, to=2-1]
    \arrow["\dashv"{anchor=center, rotate=-90}, draw=none, from=1-2, to=3-2]
  \end{tikzcd}\]

between $\Set$ and $\Cat$, where both categories are endowed with the cartesian product as the monoidal one.\footnote{We note that the counit of this adjunction is the identity natural transformation, i.e. $\pi_0(\discr(X)) = X$ for every set $X$, making $\Set$ a reflective subcategory of $\Cat$.}
The right adjoint functor $\discr$ sends a set to a discrete category, and the left adjoint $\pi_0$ sends a category to its set of connected
components.\footnote{Recall that the set of connected components of a category
  is a quotient set identifying any two objects connected by a sequence of
  arrows, where we ignore their direction.} \footnote{The
  components of the unit of this adjunction are functors $\eta_\cC : \cC \to
  \discr(\pi_0(\cC))$ with the interesting characterisation that they send every morphism to identity.}
We show that this adjunction is instrumental in mediating the connection between
1-colimits and (op)lax colimits, latter of which provide an extra level of
fidelity necessary for our purposes.
This guides us in redefining the hom-set of optics
to a hom-\emph{category} of optics, and the category of optics to a
\emph{2-}category of optics.

As colimits we're interested in are $\Set$-colimits, the situation is
straightforward. Colimits arise out of a canonical higher-dimensional version thereof: the connected components of the (op)lax colimit of the original functor.

\begin{lemma}
  Let $F : \cC \to \Set$ be a functor. Then there is an isomorphism
  \[
    \colim(F) \cong \pi_0(\colim_{\textsf{oplax}} (F \comp \discr))
  \]
  where $\pi_0 \dashv \discr$ is the adjunction between $\Set$ and $\Cat$.
\end{lemma}

This is precisely the well-known general formula for computing colimits in
$\Set$ as described in \cite[Theorem 6.37.]{SevenSketches} where the
equivalence relation described therein is the one arising as the image of
$\pi_0$.

If we are interested in having the equivalence relation be explicit as
higher-categorical cells, all we have to do is compute the oplax colimit instead
-- which corresponds to taking the Grothendieck construction of our functor.
This is precisely what we set out to do with the notion of a coend in the
formulation of optics (note that in Remark \ref{rem:eq_symmetric} we've lost track of directionality of reparameterisation morphisms, precisely because connected components ignore directionality).


%

One last thing remains: that is to recast the coend as a colimit:

\begin{restatable}{proposition}{CoendsAreColimits}
  Let $F : \cC^\op \times \cC \to \Set$ be a functor. Then we have that
  \[
    \int^X F(X, X) \cong \colim (\pi' \comp F)
  \]
  where $\pi' : \tw(\cC)^\op \to \cC^\op \times \cC \coloneqq (m \xrightarrow{r}
  n) \mapsto (n, m)$.\footnote{Equivalently, one may say that the coend of $F$ is the colimit of $F$ \emph{weighted by} the
    $\Hom_\cC$ functor. This is precisely what this proposition states: that the
  category of elements ``absorbs'' weights.}
\end{restatable}

\begin{proof}
Appendix. 
\end{proof}

A higher-categorical version of this proposition can be shown to hold.

\begin{proposition}[Oplax coends are oplax colimits]
Let $\cC$ be a category and $F : \cC^{op} \times \cC \to \Cat$ an oplax functor. Then
  \[
    \sqint^C F(C, C) \simeq \colim_{\textsf{oplax}}(\pi' \comp F)
  \]
  
\end{proposition}

This involves a routine, but a painstakingly tedious checking of the
corresponding universal properties which we thus omit.
With these two propositions in hand we can show that there is a following isomorphism

\[
 \int^C F(C, C) \cong \colim (\pi' \comp F) \cong \pi_0(\colim_{\textsf{oplax}}(\pi' \comp
 F \comp \discr)) \cong \pi_0(\sqint^C\discr(F(C, C)))
\]

motivating the definition of the hom-\emph{category} of optics as an oplax coend
in the following section.

\section{2-optics}
\label{sec:two_optics}

With the above ideas in mind, we begin to define the 2-category of optics.
Throughout this section, we fix a symmetric strict monoidal
category $\cC$.\footnote{This means we assume the associativity and unitality hold
  strictly, but not symmetry. This kind of a monoidal category is sometimes referred to as a \emph{permutative} category.}
We start by defining its hom-categories.

\begin{definition}
  \label{def:homcat_optic}
  We define the \emph{hom-category} of 2-optics $(A, A') \to (B, B')$  as the
  following oplax coend
  \[
    \TwoOptic(\cC)\OpticHom{A}{A'}{B}{B'} \coloneqq \sqint^{M : \cC} \cC(A, M \otimes B)
    \times \cC(M \otimes B', A')
  \]
  where $\cC$ is implicitly treated as a locally discrete 2-category.

  Explicitly, its objects are triples $(\M, \fw, \bw)$ (as in \ref{def:homset_optic}). A
  morphism $\omega_r : (\M_1, \fw_1, \bw_1) \to (\M_2, \fw_2, \bw_2)$ is given
  by a map $r : \M_1 \to \M_2$ such that the following diagrams commute:
  \begin{equation}
    \label{eq:twocell_optics}
  \begin{tikzcd}[row sep=3ex, column sep=4ex]
      A && {M \otimes B} && {M \otimes B'} && {A'} \\
      \\
      && {N \otimes B} && {N \otimes B'}
      \arrow["f", from=1-1, to=1-3]
      \arrow["{r \otimes B}", from=1-3, to=3-3]
      \arrow["g"', from=1-1, to=3-3]
      \arrow["{f^\sharp}", from=1-5, to=1-7]
      \arrow["{g^\sharp}"', from=3-5, to=1-7]
      \arrow["{r \otimes B'}"', from=1-5, to=3-5]
    \end{tikzcd}
  \end{equation}

  Morphisms of optics are subject to the following axioms: 
  \begin{itemize}
  \item $\omega_{\id_{\M_1}} = \id_{(\M_1, \fw_1, \bw_1)}$, and
  \item $\omega_{r \comp s} = \omega_r \comp \omega_s$ for any $r : \M_1 \to \M_2$ and $s : \M_2 \to \M_3$.
  \end{itemize}
  They tell us that a) an optic moprhism induced by the identity residual morphism
  is equal to the identity optic morphism morphism, and b) a composition of
  optic morphisms that are individually induced by residual morphisms is
  equal to the optic morphism induced by the composition of the
  aforementioned residual morphisms.
\end{definition}

\begin{remark}
Unlike with 1-optics (Remark \ref{rem:eq_symmetric}), morphisms of optics are
not quotiented out by an equivalence relation, but instead residual morphisms are refieid as explicit 2-cells.
\end{remark}

Directedness plays an important role here.
A 2-cell $\omega_r : (\M_1, \fw_1, \bw_1) \to (\M_2, \fw_2, \bw_2)$ can be interpreted operationally in a few ways:
\begin{itemize}
\item We say it \emph{moves the boundary down} from $\M_1$ to $\M_2$.
\item We say it \emph{moves the reparameterisation up} from the backward pass to the forward pass.
\item We think of the arrow as saying ``can be optimised to''. The idea is that
  the transport of the reparameterisation to the forward pass allows us us
  to statically simplify the resulting computation, potentially removing any
  redundancies.
\end{itemize}

We're finally in a position of being able to define the 2-category of optics.

\begin{definition}
We define the 2-category of optics $\TwoOptic(\cC)$ with the following data:
\begin{itemize}
\item Its objects are the same as those of $\Optic(\cC)$, i.e. pairs of objects
  in $\cC$;
\item The hom-category is defined as in Def. \ref{def:homcat_optic}, i.e.
  morphisms are optics and 2-cells are optic reparameterisations
\end{itemize}
\end{definition}

\begin{proof}
As this definition closely follows that of the 1-categorical $\Optic(\cC)$, all
we have to check is that this coherently behaves with respect to strictness.
For instance, composing three optics with residuals $\M_1, \M_2$, and $\M_3$
respectively yields an optic with residuals either $(\M_1 \otimes \M_2) \otimes \M_3$ or
$\M_1 \otimes (\M_2 \otimes \M_3)$.
But as our starting monoidal category has strict associators, these are equal.
Similar argument holds for unitality, making this a 2-category.\footnote{We
  observe that if our starting category was merely symmetric monoidal,
  $\TwoOptic(\cC)$ would be a bicategory, which is a headache we want to avoid.}
\end{proof}

This 2-categorical construction can always be turned back to the 1-categorical
one by locally quotienting things out.

\begin{proposition}
There is an isomorphism
  \[
    \pi_{0^*}(\TwoOptic (\cC)) \cong \Optic(\cC)
  \]
  where $\pi_{0^*} : \TwoCat \to \Cat$ is the enriched base change of the
  connected components functor $\pi_0$.
\end{proposition}

\begin{proof}
This is straightforward to show as $\pi_{0^*}$ is identity-on-objects, and on
the hom-category it's defined as the application of $\pi_0$.
\end{proof}

Having upgraded our optics to a 2-category, we ask whether the previously defined
equivalence between cartesian lenses and cartesian optics described in the
subsection \ref{subsec:lens_optic_iso} has a higher dimensional counterpart.
We will see that the answer is yes, and we proceed to unpack these much more involved constructions.

\subsection{Cartesian 2-optics}

In this subsection we assume the monoidal product of $\cC$ is given by the
cartesian one.
Recall that in Prop. \ref{prop:lenseqoptic} we have shown that there is a local isomorphisms between the hom-sets of lenses and optics.
As we've upgraded our optics to a 2-category, we might suspect that the
corresponding isomorphism is upgraded too.
We see that it is -- to an adjunction.

\begin{restatable}{theorem}{CartLocalAdjunction}
  \label{thm:cart_local_adjunction}
  There is an adjunction
  \[\begin{tikzcd}[ampersand replacement=\&, row sep=2ex]
      \& {} \\
      \\
      {\CartLens(\cC)\OpticHom{A}{A'}{B}{B'}} \&\& {\TwoOptic(\cC)\OpticHom{A}{A'}{B}{B'}} \\
      \\
      \& {}
      \arrow["R", curve={height=-30pt}, from=3-1, to=3-3]
      \arrow["E", curve={height=-30pt}, from=3-3, to=3-1]
      \arrow["\dashv"{anchor=center, rotate=-90}, draw=none, from=1-2, to=5-2]
    \end{tikzcd}\]
  given by
  \begin{enumerate}
  \item \emph{Residual reifier} $R$, the left adjoint functor which chooses a
    canonical residual for every lens:
    \[
      R(\prb{f}{f'}) \coloneqq (A, \graph(f), f')
    \]
    As the domain is a set, the action of $R$ on morphisms is trivial.
  \item \emph{Residual eraser} $E$, the right adjoint which normalises the optic to
    its cartesian representation:
    \[
      E((\M, \fw, \bw)) \coloneqq \prb{\fw \comp \pi_2}{(\fw \comp \pi_1) \times B' \comp \bw}
    \]
    As the codomain is a set, every morphism must be sent to the identity one.
  \item (Vacuously natural) identity transformation $\eta$, the unit of the adjunction.
  \item The natural transformation $\epsilon$, the counit of the adjunction
    whose component at each optic $(\M, \fw, \bw)$ is an optic morphism
    \[
      \epsilon_{(\M, \fw, \bw)} : R(E(\M, \fw, \bw)) \Rightarrow (\M, \fw, \bw)
    \]
    defined by the reparameterisation $\fw \comp \pi_1 : A \to \M$ in $\cC$.
  \end{enumerate}
\end{restatable}

\begin{proof}
Appendix.
\end{proof}

\begin{remark}
  As unit is the identity 2-cell (and thus an isomorphism), this means that
  $\CartLens(\cC)$ is locally a coreflective subcategory of $\TwoOptic(\cC)$. Even
  more, this is sometimes called a \emph{rali} or a \emph{lari} adjunction
  \cite[Def. 1.2.]{Lali} as the unit is the \emph{identity} 2-cell.
\end{remark}

This defines a local adjunction on hom-categories.
To show that there is some higher correspondence between the $\CartLens(\cC)$
and $\TwoOptic(\cC)$ that is not just local, we need to check whether we can use
it to define functors going both ways.
As our codomain is now a 2-category, we can reasonably expect a an (op)lax functor to appear. And having in mind the question posed in Subsection
\ref{subsec:two_ways_to_compose} we will see that our functor is indeed \emph{oplax},
as it doesn't preserve composition on the nose, instead distinguishing between the
aforementioned optics.

\begin{theorem}
  \label{thm:oplax_embedding}
There is an identity-on-objects, oplax
2-functor embedding the category of lenses into the 2-category of optics: 
\[
  \iota : \CartLens(\cC) \to \TwoOptic(\cC)
\]
defined as follows. Its action on hom-sets of $\CartLens(\cC)$ is defined in
Thm. \ref{thm:cart_local_adjunction}. Its oplaxator is the natural transformation $\delta$
\[\begin{tikzcd}
	{\CartLens(\cC)\OpticHom{A}{A'}{B}{B'} \times\CartLens(\cC)\OpticHom{B}{B'}{C}{C'}} && {\CartLens(\cC)\OpticHom{A}{A'}{C}{C'}} \\
	\\
	\\
	{\TwoOptic(\cC)\OpticHom{A}{A'}{B}{B'} \times\TwoOptic(\cC)\OpticHom{B}{B'}{C}{C'}} && {\TwoOptic(\cC)\OpticHom{A}{A'}{C}{C'}}
	\arrow["{\text{Compose}}", from=1-1, to=1-3]
	\arrow["{\substack{\text{Turn into}\\ \text{optics}}}", from=1-3, to=4-3]
	\arrow["{\substack{\text{Turn into}\\ \text{optics}}}"', from=1-1, to=4-1]
	\arrow["{\text{Compose}}"', from=4-1, to=4-3]
	\arrow["{\delta^\iota}"{description}, shorten <=20pt, shorten >=20pt, Rightarrow, from=1-3, to=4-1]
\end{tikzcd}\]

which to every pair of lenses $\prb{A}{A'} \xrightarrow{\prb{f}{f'}}
\prb{B}{B'}$ and $\prb{B}{B'} \xrightarrow{\prb{g}{g'}} \prb{C}{C'}$ assigns the
reparameterisation $\graph(f) : A \to A \times B$ between the corresponding
optics $\iota(\prb{f}{f'} \comp \prb{g}{g'})$ and $\iota(\prb{f}{f'}) \comp
\iota(\prb{g}{g'})$.
Its opunitor to each object $\prb{A}{A'}$ assigns a natural
transformation $\epsilon^\iota$ whose unique component 2-cell $\epsilon^\iota :
\iota(\id_{\tiny\prb{A}{A'}}) \Rightarrow \id_{\tiny{\iota\prb{A}{A'}}}$ is
given by the reparameterisation $!_A : A \to 1$.
\end{theorem}

\begin{proof}
We first need to prove that $\graph(f)$ is a well-defined reparameterisation
between $\iota(\prb{f}{f'} \comp \prb{g}{g'})$ and $\iota(\prb{f}{f'}) \comp
\iota(\prb{g}{g'})$.
These two unpack to two optics previously drawn in Figure \ref{fig:lens_optic_comp}:
  \begin{align*}
    \iota(\prb{f}{f'} \comp \prb{g}{g'}) \quad 	&\leadsto \quad (A, \textcolor{black}{\graph(f \comp g), (\graph(f) \times C') \comp (A \times g') \comp f'})\\
    \iota(\prb{f}{f'}) \comp \iota(\prb{g}{g'}) \quad &\leadsto \quad (A \times B, \textcolor{black}{\graph(f) \comp (A \times \graph(g)), (A \times g') \comp f'})
  \end{align*}

  Once reparameterised, we can see that the residuals and backward maps of these
  optics are equal on the nose.
  This leaves us with just proving that the forward parts are.
  We postpone this proof to Remark \ref{rem:rewrite} as it has special meaning
  in terms of rewrites.
  From the naturality of the delete map it is straightforward to prove that the opunitor reparameterisation $!_A$ is well-defined too.

  Lastly, what needs to be proven is that lax associativity and lax left and right unity
  are satisfied (\cite[Eq. 4.1.19 and 4.1.20]{TwoDimensionalCategories}).
  This becomes evident once the diagrams are properly unpacked, and simple
  axioms of a cartesian category (associativity and naturality of copy,
  interaction of copy and delete) are sufficient to prove it.
\end{proof}

\begin{remark}
  \label{rem:rewrite}
The attentive reader might have noticed a peculiarity regarding the two possible
ways of composing lenses into optics.
Namely, reparameterising the forward part $\graph(f \comp g)$ of the optic
$\iota(\prb{f}{f'} \comp \prb{g}{g'})$ with $graph(f) : A \to A \times B$ yields
the forward part $\graph(f \comp g) \comp (\graph(f) \times C)$.
To show that this morphism is indeed equal to the forward part $\graph(f) \comp
(A \times \graph(g))$ of the optic $\iota(\prb{f}{f'}) \comp
\iota(\prb{g}{g'})$ we need to exhibit an additional proof.

\begin{figure}[H]
  \centering
  \includegraphics[width=0.7\textwidth]{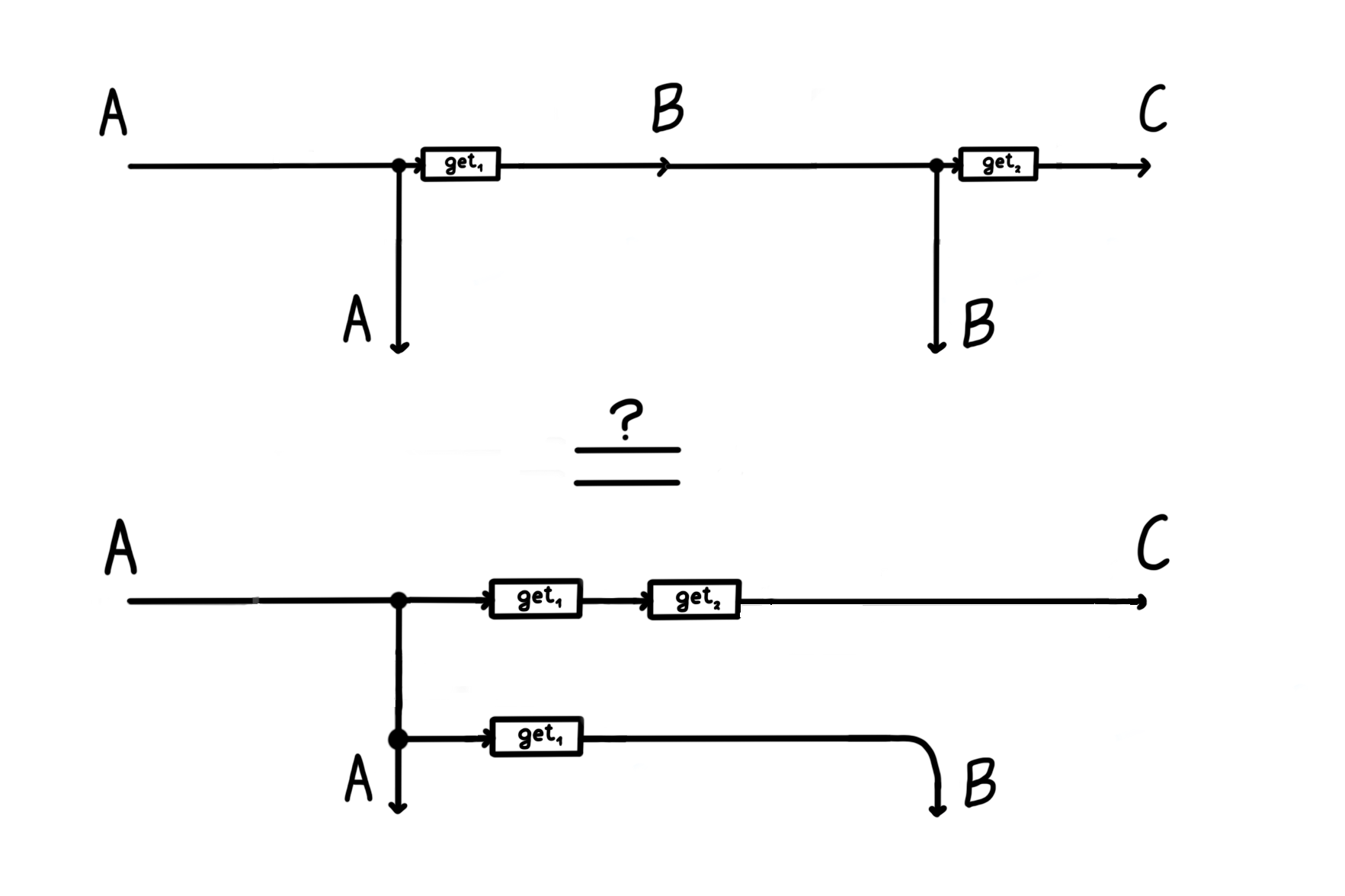}
  \caption{Two equivalent optic forward parts: $\graph(f) \comp (A \times
    \graph(g))$ (top) and $\graph(f \comp g) \comp (\graph(f) \times C)$ (bottom).}
  \label{fig:cart_rewrite}
\end{figure}

This proof is rather simple as the rewrite can be made by applying
associativity and naturality of the copy map.
But surprisingly, this proof is the main reason why optics can trade
space for time.
This is because we can compute the $\view_1$ once, and the copy the result instead of
copying the input and applying $\view_1$ twice, separately.
\end{remark}

\begin{remark}
  \label{rem:closed_lax}
  We do not provide a proof in this paper, but \emph{closed} lenses exhibit
  lax structure when embedded into 2-optics.
  However, they happen not to permit any additional rewrites as a result of
  this.
  This means that their operational characteristics end up the same as those of
  optics.
  This might be the reason why essentially closed lenses are what's used in
  \cite{SimpleAD} and many automatic differentiation libraries, as function
  spaces are a feature in most languages, but dependent types, which are
  required for optics, are not. 
\end{remark}

Having defined the embedding of lenses into optics, we can also go the other
way. This time, we have a strict 2-functor.

\begin{proposition}
There is a strict identity-on-objects 2-functor $\TwoOptic(\cC)
\xrightarrow{\pi_*} \CartLens(\cC)$ normalising an optic to its lens
representation whose action on hom-objects is defined in Thm \ref{thm:cart_local_adjunction}.
\end{proposition}

\subsection{What world do these constructions live in?}

So far we've had an obstacle-free path retelling the story of optics in the
2-categorical language.
We have defined the 2-category of optics and showed that, in the cartesian case,
there's an oplax functor going to lenses, and a strict 2-functor going backwards.
As their lower-dimensional version formed an isomorphism, the natural next step
is to check whether these two weak functors form some kind of a weaker version of an isomorphism -- such as an equivalence or
an adjunction.
This necessitates finding a suitable ambient 2-category, and is where peculiarities start appearing.

We first recall all the data defined so far, and some relevant properties:

\begin{itemize}
\item Our 0-cells are 2-categories;
\item Our 1-cells are not all strict;
\item Our 1-cells $\iota$ and $\pi_*$ are \emph{identity-on-objects};
\item Thm. \ref{thm:cart_local_adjunction} establishes a local hom-category
  \emph{adjunction}, and not a mere isomorphism.
\end{itemize}

Looking at only the first three bullet points, we would be lead down a rabbit hole
that eventually proves to be a dead end.
As the second bullet point states that our 1-cells are not all strict, this rules out the 2-category $\TwoCat$ of 2-categories, 2-functors and lax transformations as a suitable candidate.
This means we need a 2-category whose 1-cells are lax functors. Famously, such a
2-category where 2-cells are any of the usual strict/pseudo/(op)lax/
transformations actually doesn't exist \cite{ProblemWithLaxFunctors}.
What does exist is a 2-category $\TwoCatic$ of 2-categories, lax functors and a
restricted kind of an oplax natural transformation called an \emph{icon}
(\cite{Icon}, \cite[Theorem 4.6.13.]{TwoDimensionalCategories}).
Icons are oplax natural transformations that are defined \emph{only when the
  underlying lax functors agree on objects}.
However, we see from the third bullet point that this is indeed the case for us!
We can indeed define these icons, and we will see that one of them is identity.
This makes one of the triangle identities commute automatically. However,
problems arise when when we check the other one -- we find that it doesn't commute.

This is because of the fourth bullet point: an adjunction requires a local
isomorphism, but we locally have an adjunction itself. This rules out the
possibility of an adjunction internal to a 2-category -- meaning we have to search for a weaker 2-categorical analogue thereof.
This leads us to consider the notion of a \emph{lax 2-adjunction} instead,
defined internal to a \emph{3-category}.
This is a concept weak enough in for our purposes, describing exactly the
setting of a local adjunction between hom-categories.

This is where our search, in its current form, stops.
Even though there is a special restricted kind of 2-category $\TwoCatic$ whose
2-cells are icons, there truly is no \emph{3-category} whose 0-cells are
2-categories and 1-cells are lax functors. Thus the question as we've posed it
indeed has no answers.

On the other hand, we might want to pose a different question. 
All of our constructions are restricted in some ways -- 1-cells are
identity on objects and 2-cells have identity components.
Perhaps our approach is failing because 0-cells are also restricted, but we're
not looking at them from the right perspective?

This is indeed the case. It turns out that there is an embedding $\TwoCatic \to
\Dbl$ where $\Dbl$ is the 3-category of double categories, lax functors, lax
transformations and modifications.
It sends a 2-category to a vertically discrete double category, a lax functor to
a lax functor between the corresponding double categories and an icon to
vertical transformations.

This means that our 2-category of optics should really be thought of as a double
category with only trivial vertical arrows.
As $\Dbl$ is a 3-category this means that it is a plausible setting for defining
a lax 2-adjunction.
This leaves us with the conjecture that the well-known isomorphism between
$\CartLens(\cC)$ and $\Optic(\cC)$ is a shadow of the lax 2-adjunction in the
3-category $\Dbl$ between $\CartLens(\cC)$ and $\TwoOptic(\cC)$ appropriately
thought of as double categories.
The proof of this conjecture is something we leave to future work.

	\bibliographystyle{alpha}
	\bibliography{bibliography}

  \appendix
  \section{Appendix}

\CartLocalAdjunction*
\begin{proof}
  We first prove well-definedness of the residual eraser $E$.
  As $E$ maps every optic morphism $r : o_1 \Rightarrow o_2$ to identity, we
  have to show that $Eo_1 = Eo_2$.
  But here we use the existing equivalence between lenses and the \emph{1-category} of optics (Prop. \ref{prop:lenseqoptic}) -- where these
lenses are equivalent if and only if the corresponding 1-optics are. That is
indeed true, as $r$ is a witness to such equivalence.

We now prove that both $\eta$ and $\epsilon$ is well-defined. For $\eta$ we
notice that starting from a lens, reifying its residual and then erasing it
lands us back \emph{exactly} with the starting lens. It's easy to see that this
assignment is vacuously natural as $\CartLens(\cC)\OpticHom{A}{A'}{B}{B'}$
is discrete.
For $\epsilon$, we can see that starting from an optic $(\M, \fw, \bw)$, erasing
the residual (yielding $\prb{\fw \comp \pi_2}{(f \comp \pi_1) \times B' \comp
  f'}$) and then reifying the residual results in an optic $(A, \graph(\fw \comp
\pi_2), (\fw \comp \pi_1) \times B' \comp \bw)$.
We need to check whether $\fw \comp \pi_1$ is a well-defined optic 2-cell:
\begin{equation}
  \label{eq:cart_counit_roundtrip}
  (A, \graph(\fw \comp \pi_2), (\fw \comp \pi_1) \times B' \comp \bw) \xRightarrow{\fw \comp \pi_1} (\M, \fw, \bw)
\end{equation}

But this is easy to verify. As optic 2-cells move reparameterisations backward,
the place where we have to look for a suitable reparameterisation of type $A \to
M$ is precisely in the backward pass of the left-hand side. And it is indeed there -
it's the morphism $\fw \comp \pi_1$.
We next need to prove that $\epsilon$ is natural. This means that for every
optic morphism $r : (\M_1, \fw_1, \bw_1) \to (\M_2, \fw_2, \bw_2)$ the equation
$\epsilon_{(\M_1, \fw_1, \bw_1)} \comp r = R(E(r)) \comp \epsilon_{(\M_2, \fw_2,
  \bw_2)}$ needs to hold.
As $R(E(r))$is identity, this reduces to showing that $\fw \comp \pi_1 \comp r =
\fw_2 \comp \pi_1$ which follows from Eq. \ref{eq:twocell_optics} (left).

Lastly, to prove this data indeed defines an adjunction, we need to show that
the following diagrams commute.

\[\begin{tikzcd}[ampersand replacement=\&]
    R \&\& {R \comp E \comp R} \&\& E \&\& {E \comp R \comp E} \\
    \\
    {} \&\& R \&\&\&\& E
    \arrow[Rightarrow, no head, from=1-1, to=3-3]
    \arrow["{R \comp \epsilon}", from=1-3, to=3-3]
    \arrow["{\eta \comp R}", Rightarrow, from=1-1, to=1-3]
    \arrow["{E \comp \eta}", Rightarrow, from=1-5, to=1-7]
    \arrow["{\epsilon \comp E}", Rightarrow, from=1-7, to=3-7]
    \arrow[Rightarrow, no head, from=1-5, to=3-7]
  \end{tikzcd}\]

As $\eta$ is the identity natural transformation, this means that $\eta \comp R = R$ and $E \comp \eta = E$, reducing the proofs to $R \comp \epsilon = \id_R$ and $\epsilon \comp E = \id_E$, respectively.
For the former, we need to show that applying the counit from Eq.
\ref{eq:cart_counit_roundtrip} on $(A, \graph(f), f')$ naturally the identity
morphism. This is easy to show as the underlying reparameterisation morphism
$\fw \comp \pi_1$ in for $\fw \coloneqq \graph(f)$ via Prop.
\ref{prop:graph_properties} reduces to identity.
For the latter, we need to show that applying $E$ to the same counit yields
identity. This is also straightforward as $E$ maps every morphism to identity.
\end{proof}

\begin{definition}[Ends as limits]
Let $\cC$ be a category. We call $\tw(\cC)$ the twisted arrow category of $\cC$
defined as the category of elements of its hom functor.
\[
  \tw(\cC) \coloneqq \El(\Hom_\cC)
\]

It comes equipped with the projection $\pi_{\cC} : \tw(\cC) \to \cC^\op \times \cC$.
\end{definition}

\begin{proposition}\label{prop:tw_iso}
There is a canonical isomorphism $\tw(\cC) \cong \tw(\cC^\op)$.
\end{proposition}





\CoendsAreColimits*

\begin{proof}
  \begin{align*}
          &\int^C F(C, C)\\
    \cong & \quad (\text{Duality of ends and coends})\\
          &\int_C F^\op(C, C)\\
    \cong & \quad (\text{Ends as limits})\\
          &\lim (\tw(\cC^\op) \xrightarrow{\pi_{\cC^\op}} \cC \times \cC^\op \xrightarrow{F^\op} \Set^\op)\\
    \cong & \quad (\text{Duality of limits and colimits})\\
          & \colim (\tw(\cC^\op)^\op \xrightarrow{\pi_{\cC^\op}^\op} \cC^\op \times \cC \xrightarrow{F} \Set)\\
    \cong & \quad (\text{Prop. \ref{prop:tw_iso}})\\
          & \colim (\tw(\cC)^\op \xrightarrow{\pi_{\cC}'} \cC^\op \times \cC \xrightarrow{F} \Set)\\
  \end{align*}
  
\end{proof}

\begin{definition}\label{def:graph}
Let $f : A \to B$ be a morphism in cartesian category $\cC$. Then we denote by
$\graph(f)$ the composite
\[
  A \xrightarrow{\Delta_A} A \times A \xrightarrow{A \times f} A \times B
\]

This is called the \emph{graph of f}.\footnote{It's called the \emph{graph} of f because its image is a set of pairs $(a, f(a))$ which we can think of as points in a coordinate plane to be graphed.}
\end{definition}

\begin{proposition}\label{prop:graph_properties}
  Let $f : A \to B$  in a cartesian category $\cC$. Then we have that
  \begin{align*}
    \graph(f) \comp \pi_2 &= f, \quad \text{and}\\
    \graph(f) \comp \pi_1 &= \id_A
  \end{align*}
\end{proposition}

\end{document}